\documentclass[a4paper]{article}			
\usepackage[latin1]{inputenc}				
\usepackage[T1]{fontenc}					
\usepackage{lmodern}						
\usepackage[english]{babel}					
\usepackage{amsmath,amsfonts}				
\usepackage{amssymb}						
\usepackage{stmaryrd}                       
\usepackage{multirow}						

\usepackage{bm}								
\usepackage{amsthm}

\usepackage{a4wide} 						

\newtheorem{thm}{Theorem}
\newtheorem{lem}{Lemma}

\newtheorem{conj}{Conjecture}

\newtheorem{qst}{Question}

\theoremstyle{definition}					
\newtheorem{rem}{Remark}

\newcommand{\R}{\mathbb{R}}

\newcommand{\N}{\mathbb{N}}
\newcommand{\Z}{\mathbb{Z}}
\newcommand{\Q}{\mathbb{Q}}

\title{A note on the Hausdorff dimension of some liminf sets appearing in simultaneous Diophantine approximation}
\author{Faustin ADICEAM\\
   Department of Mathematics, Logic House,\\ 
   National University of Ireland at Maynooth\\
   email~:\texttt{fadiceam@gmail.com}}
\date{}

\begin{document}
\maketitle

\begin{abstract}
Let $Q$ be an infinite set of positive integers. Denote by $W_{\tau, n}(Q)$ (resp. $W_{\tau, n}$) the set of points in dimension $n\ge 1$ simultaneously $\tau$--approximable by infinitely many rationals with denominators in $Q$ (resp. in $\N^*$). A non--trivial lower bound for the Hausdorff dimension of the liminf set $W_{\tau, n}\backslash W_{\tau, n}(Q)$ is established when $n\ge 2$ and $\tau >1+1/(n-1)$ in the case where the set $Q$ satisfies some divisibility properties. The computation of the actual value of this Hausdorff dimension as well as the one--dimensional analogue of the problem are also discussed.
\end{abstract}

Let $n \ge 1$ be an integer and let $\tau > 1$ be a real number. Consider the classical set of simultaneously $\tau$--approximable points in dimension $n$, viz. $$W_{\tau, n} := \left\{ (x_1, \dots, x_n)\in \R^n \; : \; \forall i\in \llbracket 1, n \rrbracket,    \left|x_i-\frac{p_i}{q}\right|<\frac{1}{q^{\tau}} \; \; \mbox{ for  i.m. } q\ge 1 \right\}.$$ Here and in what follows, \textit{i.m.} $q\ge 1$ stands for \textit{infinitely many integers $q\ge 1$ associated to fractions $p_1/q, \dots, p_n/q$} and $\llbracket 1, n \rrbracket$ denotes the interval of integer $\{1, \dots, n\}$. The Hausdorff dimension of a set $X \subset \R^n$ shall be denoted by $\dim X$.

From the multidimensional generalization of the theorem of Jarn\'ik and Besicovitch (see for instance~\cite{multijar}), the Hausdorff dimension of the set $W_{\tau, n}$ is known to be 
\begin{equation}\label{jarbes}
\dim W_{\tau, n} = \frac{n+1}{\tau}
\end{equation} 
as soon as $\tau > 1+1/n$. Borosh and Fraenkel~\cite{broshfraenkel} have extended this result in the following way~: for any infinite set $Q$ of positive natural numbers, define the exponent of convergence $\nu(Q)$ of $Q$ as 
\begin{equation*}
\nu(Q) := \inf \left\{\nu >0 \; : \; \sum_{q\in Q} q^{-\nu} < \infty \right\} \in [0,1].
\end{equation*}
Then, the Hausdorff dimension of the set $$W_{\tau, n}(Q) := \left\{ (x_1, \dots, x_n)\in \R^n \; : \; \forall i\in \llbracket 1, n \rrbracket,    \left|x_i-\frac{p_i}{q}\right|<\frac{1}{q^{\tau}} \; \; \mbox{ for i.m. } q\in Q \right\}$$ is 
\begin{equation}\label{borofraen}
\dim W_{\tau, n}(Q) = \frac{n+\nu(Q)}{\tau}
\end{equation}
when $\tau > 1+ \nu(Q)/n$. For recent developments and related results see~\cite{glylivre} (in particular Chapters 6 and 10), \cite{dickirynne} and \cite{rynne}.

This paper is concerned with some properties of the complement of $W_{\tau, n}(Q)$, namely, $$W_{\tau, n}^*(Q) := W_{\tau, n} \backslash W_{\tau, n}(Q);$$ that is, 
\begin{equation}\label{wtnq}
W_{\tau, n}^*(Q) := \left\{ (x_1, \dots, x_n)\in \R^n \; : \; \forall i\in \llbracket 1, n \rrbracket,    \left|x_i-\frac{p_i}{q}\right|<\frac{1}{q^{\tau}} \; \; \mbox{ for  i.m. } q \not\in Q \mbox{ and f.m. } q\in Q \right\},
\end{equation}
where \textit{f.m.} stands for \textit{finitely many}. To the best of the author's knowledge, the actual Hausdorff dimension of the liminf set $W_{\tau, n}^*(Q)$ is unknown in the general case. This seems to be a very difficult problem which shall be discussed here in the case where $Q$ is a set defined by divisibility properties. 

More precisely, let $Q$ be any infinite set of positive integers satisfying the following condition~: for any integer $q\in \N^*$, 
\begin{equation}\label{defQ}
\left(q\in Q \right) \; \Leftrightarrow \; \left(\forall v\not\in Q, \: v\!\!\!\not| q \right).
\end{equation} 

Such a set $Q$ shall be called an \textit{$\N^*\backslash Q$--free set} following the definition of a $B$--free set introduced by Erd\H{o}s~\cite{erdos} \footnote{However, unlike Erd\H{o}s, the convergence of the series $\sum_{q\in \N^*\backslash Q} q^{-1}$ is not required here. Notice that a set $Q$ which is $\N^*\backslash Q$--free must contain 1.}. Examples of sets satisfying this property are the square--free integers (or, more generally, the $k$--free integers with $k\ge 2$ --- conventionally, one is considered here as a $k$--free integer for any $k\ge2$) or the set of integers coprime to a given natural number $m\neq 1$. 

For the definition of $W_{\tau, n}^*(Q)$ to make sense, it is natural to impose the condition $Q \varsubsetneq \N^*$, which shall be assumed throughout. It is indeed easily seen that such a condition ensures that the complement set $\N^*\backslash Q$ is infinite.

An elementary property of the Hausdorff dimension (see for instance~\cite{berdod}, p.65) leads to the relationship 
\begin{equation*}
\dim W_{\tau, n} = \max \left\{\dim W_{\tau, n}(Q),\, \dim W_{\tau, n}^*(Q) \right\}.
\end{equation*}
It follows from~(\ref{jarbes}) and~(\ref{borofraen}) that $\dim W_{\tau, n}^*(Q) = (n+1)/\tau$ as soon as $\nu(Q)<1$ and $\tau > 1+1/n$. In the case $\nu(Q)=1$ however, the situation is much less well understood and the computation of the actual Hausdorff dimension of the set $W_{\tau, n}^*(Q)$ probably involves some deep results on the distribution of $B$--free numbers which are still in the state of conjecture, as for instance those mentioned by Erd\H{o}s in~\cite{erdos}. It is nevertheless always possible to find a non--trivial lower bound for $\dim W_{\tau, n}^*(Q)$ if $n\ge 2$. 

\begin{thm}\label{princi}
Let $Q$ be an infinite $\N^*\backslash Q$--free set. Assume that $n\ge 1$. Then 
$$
\dim W_{\tau, n}^*(Q) \left\{
    \begin{array}{ll}
        =(n+1)/\tau & \mbox{ if }\; \nu(Q)<1,\; \tau > 1+1/n \;\mbox{ and }\; n\ge 1 \\
        \in \left[n/\tau , (n+1)/\tau \right]& \mbox{ if } \;\nu(Q) = 1,\; \tau >1+1/(n-1) \;\mbox{ and }\; n\ge 2.
    \end{array}
\right.
$$
\end{thm} 

\begin{rem}
If $Q$ is any $\N^*\backslash Q$--free set, define its $\mathrm{support}$ $\mathrm{Supp}(Q)$ as the set of all primes dividing at least one element in $Q$. It is readily seen that, if the support of $Q$ is finite, then $\nu(Q)=0$. 

On the other hand, if $\mathcal{S}$ is an infinite set of positive integers such that the series $\sum_{s\in \mathcal{S}} s^{-\mu}$ diverges for some $\mu >0$, one may construct by the diagonal process a subset $\mathcal{S'}$ of $\mathcal{S}$ such that the series $\sum_{s\in \mathcal{S'}} s^{-\mu}$ diverges and the series $\sum_{s\in \mathcal{S'}} s^{-\mu-\epsilon}$ converges for any $\epsilon >0$ (that is, a subset $\mathcal{S'}$ such that $\nu(S')=\mu$). Since the series of the reciprocal of the primes is divergent, this shows that, for any $\alpha \in (0,1]$, there exists an $\N^*\backslash Q$--free set $Q$ such that $\nu(Q)=\alpha$.

It should also be noticed that the exponent of convergence of an $\N^*\backslash Q$--free set $Q$ depends only on its support. Indeed, for any such set, it should be clear that 
\begin{equation}\label{inegseries}
\sum_{\pi\in\mathrm{Supp}(Q)}\frac{1}{\pi^{\nu}} \; \le \; \sum_{q\in Q} \frac{1}{q^{\nu}} \; \le \; \sum_{n\in \mathcal{C}\left(\mathrm{Supp}(Q) \right)}\frac{1}{n^{\nu}},
\end{equation}
where $\nu >0$ and $\mathcal{C}\left(\mathrm{Supp}(Q) \right)$ is the set of positive integers all of whose prime factors belong to $\mathrm{Supp}(Q)$. The series on the right--hand side of~(\ref{inegseries}) admits an Euler product expansion given by
\begin{equation*}
\sum_{n\in \mathcal{C}\left(\mathrm{Supp}(Q) \right)}\frac{1}{n^{\nu}} = \prod_{\pi\in\mathrm{Supp}(Q)} \left(1+ \frac{1}{\pi^{\nu}-1} \right).
\end{equation*}
Taking the logarithm of this last quantity, it is readily seen that the right--hand side of~(\ref{inegseries}) converges if, and only if, the left--hand side of~(\ref{inegseries}) converges, hence the result.
\end{rem}

From Theorem~\ref{princi}, as soon as $\tau > 1+1/(n-1)$ (where $n\ge 2$), the inequality $$\dim W_{\tau, n}^*(Q) \ge \frac{n}{\tau}$$ holds true  as it does for the set $W_{\tau, n}(Q)$ when $\tau > 1+ \nu(Q)/n$ from~(\ref{borofraen}). The result of the theorem also implies that, for any fixed $\tau > 1$, $$ \dim W_{\tau, n}^*(Q) \underset{n \rightarrow +\infty}{\sim} \dim W_{\tau, n}(Q).$$ 

\paragraph{}
The proof of Theorem~\ref{princi} is based on the following lemma, which states that good simultaneous rational approximations to given rationally dependent real numbers must satisfy the same rational dependence relationship as do the given numbers.

\begin{lem}\label{fonda}
Let $n\ge 1$ be an integer and $\tau >1$ be a real number. Let $(x_1, \dots, x_n)\in \R^n$ be such that there exist integers $a_1, \dots, a_n$ and $b$ satisfying $$\sum_{i=1}^{n}a_i x_i = b.$$ Assume furthermore that, for all $i \in \llbracket 1, n  \rrbracket$, $$\left| x_i - \frac{p_i}{q} \right| < \frac{1}{q^{\tau}},$$ where $p_1/q, \dots, p_n/q$ are rational numbers.

Then, if $q$ is large enough (depending only on the integers $a_1, \dots, a_n$ and on the real number $\tau$), $$\sum_{i=1}^{n} a_i \frac{p_i}{q} = b.$$ 
\end{lem}

\begin{proof}
The proof is straightforward~: notice that 
\begin{equation*}
\left|qb - \sum_{i=1}^{n} a_i p_i \right| = \left| \sum_{i=1}^{n} a_i(q x_i - p_i) \right| \le \frac{\sum_{i=1}^{n} \left|a_i \right|}{q^{\tau - 1}}\cdotp
\end{equation*}
Thus if $q$ is large enough, the left hand side of the above inequality is an integer with absolute value less than one and so vanishes.
\end{proof}

For any $A = (a_1, \dots, a_n)\in\N^{n}\backslash \left\{0 \right\}$ (where $n\ge 2$) and any $u,v \in \N^*$, denote by $\Gamma_n(A,u,v)$ the rational hyperplane 
\begin{equation}\label{hyperplan}
\Gamma_n(A,u,v) := \left\{(x_1, \dots, x_n)\in \R^n \; : \; \sum_{i=1}^{n}a_i x_i = \frac{u}{v} \right\}.
\end{equation}
Even if it means relabeling the axes, it shall always be assumed, without loss of generality, that $a_n \neq 0$.

For an infinite $\N^*\backslash Q$--free set $Q$, choose $v\in \N^*\backslash Q$ and $u\in\N^*$ coprime to $v$. Then Lemma~\ref{fonda} implies that 
\begin{equation}\label{inclus}
\Gamma_n(A,u,v) \cap W_{\tau, n}  \subset W_{\tau, n}^*(Q)
\end{equation} 
as soon as $\tau >1$. Thus to prove Theorem~\ref{princi} in the case where $\nu(Q)=1$, it suffices to establish that the dimension of the set $\Gamma_n(A,u,v) \cap W_{\tau, n}$ is the expected dimension of the set of $\tau$--well approximable points in dimension $n-1$ ($n\ge 2$).

\begin{lem}
Let $n\ge 2$ be an integer and let $\Gamma_n(A,u,v)$ be a rational hyperplane in $\R^n$ as defined by~(\ref{hyperplan}). Then for $\tau> 1+1/(n-1)$, $$\dim \left( \Gamma_n(A,u,v) \cap W_{\tau, n} \right) = \frac{n}{\tau}\cdotp$$
\end{lem}

\begin{proof}
The upper bound for the Hausdorff dimension of the limsup set $\Gamma_n(A,u,v) \cap W_{\tau, n}$ may be computed with a standard covering argument. As for the lower bound, notice that Lemma~\ref{fonda} implies that 
\begin{equation*}
\Gamma_n(A,u,v) \cap W_{\tau, n} = \left\{(x_1, \dots, x_n) \in\R^n \; : \; 
\begin{split}
&  \forall i\in \llbracket 1, n \rrbracket, \,\,   \left|x_i-\frac{p_i}{q}\right|<\frac{1}{q^{\tau}} \\
& \mbox{for  i.m. } \left(\frac{p_1}{q}, \dots, \frac{p_n}{q}\right)\in\Gamma_n(A,u,v)
\end{split}
\right\}.
\end{equation*}
\sloppy For any vector $\bm{x} = (x_1,\dots,x_n)\in\R^n$ ($n\ge 2$), denote by $\bm{x_{n-1}}$ the subvector $\bm{x_{n-1}} := (x_1, \dots, x_{n-1})$ and by $\|\bm{x}\|_{\infty}$ the infinity norm of $\bm{x}$, viz. $\|\bm{x}\|_{\infty} = \underset{1\le i \le n}{\max} \left|x_i \right|$. Taking $K(A)= \left(\sum_{i=1}^{n}\left|a_i \right| \right)\left|a_n \right|^{-1}$, it is then readily verified that 
\begin{equation}\label{factice}
\|\bm{x_{n-1}} - \bm{y_{n-1}}\|_{\infty} \le \|\bm{x} - \bm{y}\|_{\infty} \le K(A)\|\bm{x_{n-1}} - \bm{y_{n-1}}\|_{\infty}
\end{equation}
for any $\bm{x}, \bm{y} \in \Gamma_n(A,u,v)$. Since the Hausdorff dimension of a set is invariant under a bi--Lipschitz transformation, this means that $$\dim \left( \Gamma_n(A,u,v) \cap W_{\tau, n} \right) = \dim V_{\tau}\left(\Gamma_n(A,u,v) \right),$$ where $$V_{\tau}\left(\Gamma_n(A,u,v) \right) = \left\{ (x_1, \dots, x_{n-1})\in\R^{n-1} \; : \; (x_1, \dots, x_{n-1}, x_n)\in \Gamma_n(A,u,v) \cap W_{\tau, n} \right\}.$$ \sloppy Notice that, if $(x_1, \dots, x_{n-1})\in \R^{n-1}$, then $x_n$ as appearing in the definition of $V_{\tau}\left(\Gamma_n(A,u,v) \right)$ is uniquely determined.

From~(\ref{factice}), it should be clear that the set $$U_{\tau, n-1}\left(K(A) \right) := \left\{(x_1, \dots, x_{n-1})\in\R^{n-1} \; : \; \forall i\in \llbracket 1, n-1 \rrbracket, \,   \left|x_i-\frac{p_i}{q}\right|<\frac{K(A)^{-1}}{q^{\tau}} \mbox{ for i.m. } q\ge 1 \right\}$$ is contained in $V_{\tau}\left(\Gamma_n(A,u,v) \right)$. To conclude the proof of the lemma, it suffices now to invoke Jarn\'ik's Theorem on simultaneous approximation~\cite{multijar} (Theorem~3) which may also be found in~\cite{linfor} (Theorem~1)~: in the latter reference, the choices $f(r)=r^s$, $\psi(r) = r^{-\tau}K(A)^{-1}$, $m=1$ and $n=n-1$ lead to the equality $\dim U_{\tau, n-1}\left(K(A) \right) = n/\tau$ if $\tau > 1+1/(n-1)$, the $n/\tau$--Hausdorff measure of  $U_{\tau, n-1}\left(K(A) \right)$ being infinite. Note that this also holds true for the sets $V_{\tau}\left(\Gamma_n(A,u,v) \right)$ and $\Gamma_n(A,u,v) \cap W_{\tau, n}$ since a set of infinite Hausdorff measure is transformed under a bi--Lipschitz transformation into a set of infinite Hausdorff measure.
\end{proof}
This completes the proof of Theorem~\ref{princi}.

\paragraph{}
Theorem~\ref{princi} does not give any information about what happens if $n=1$ and $\nu(Q)=1$. This situation is classical in Diophantine approximation~: some problems are easier to apprehend in higher dimensions than in dimension one (see for instance the conjecture of Duffin and Schaeffer). Nevertheless, it is sometimes possible to prove that the set $W_{\tau,1}^*(Q)$ is not empty, which is far from being obvious at first sight in the general case.

To illustrate this fact, consider a positive integer $m\neq 1$ and the $\N^*\backslash Q(m)$--free set 
\begin{equation}\label{qm}
Q(m):= \left\{ n\in\N^* \; : \; \gcd(n,m)=1 \right\}.
\end{equation} 
Assume that the integer $m$ is divisible by exactly $r$ prime numbers $\pi_1, \dots, \pi_r$ and set $\Pi := \left\{\pi_1, \dots, \pi_r \right\}$. For simplicity, let $W_{\tau}^*(\Pi)$ denote the set $W_{\tau,1}^{*}(Q(m))$, that is, 
\begin{equation}\label{wpi}
W_{\tau}^*(\Pi) = \left\{x\in\R \; : \; \left|x-\frac{p}{q} \right|< \frac{1}{q^{\tau}} \mbox{ for i.m. } q\not\in Q\left(\prod_{\pi\in\Pi}\pi \right) \mbox{ and f.m. } q \in Q\left(\prod_{\pi\in\Pi}\pi \right) \right\}.
\end{equation}
 
Thus an element $x\in W_{\tau}^*(\Pi)$ is $\tau$--well approximable by fractions whose denominators, except for a finite number of them, are necessarily divisible by a prime in $\Pi$. It shall be proved that $W_{\tau}^*(\Pi)$ is never empty as soon as $\tau >2$ and the cardinality of $\Pi$ is greater than or equal to 2.

To this end, the theory of continued fractions is needed~: if necessary, the reader is referred to~\cite{bug} (Chapter~1) for an account on the topic. Given a irrational number $x$, denote by $\left(a_s\right)_{s\ge 0}$ the sequence of its partial quotients and by $\left(p_s/q_s \right)_{s\ge 0}$ the sequence of its convergents, which are given by the recurrences
\begin{align}
p_s &= a_s p_{s-1}+p_{s-2} \label{pn}\\
q_s &= a_s q_{s-1}+q_{s-2} \label{qn}
\end{align}
for $s\ge 1$ along with the initial conditions $p_{-1}=1$, $p_0=a_0$, $q_{-1}=0$ and $q_0=1$. The convergents of $x\in\R\backslash\Q$ are related to its rational approximations in the following way~: if a non--zero rational number $p/q$ satisfies the inequality 
\begin{equation}\label{approx}
\left| x-\frac{p}{q}\right| < \frac{1}{2q^2},
\end{equation}
then $p/q$ is a convergent of $x$, i.e. there exists $s\in\N$ such that $p/q = p_s/q_s$.

The following theorem may now be proved~:

\begin{thm}\label{nonempty}
Let $\Pi$ be any subset of the primes containing at least two distinct prime numbers. Let $\tau >2$ be a real number. Then the set $W_{\tau}^*(\Pi)$ as defined by~(\ref{wpi}) contains uncountably many Liouville numbers.
\end{thm}

Theorem~\ref{nonempty} shall be derived from the following result, proved by Erd\H{o}s and Mahler~\cite{erdmah} (Theorem~2). The notation introduced above is kept~:

\begin{thm}[Erd\H{o}s \& Mahler, 1939]\label{erdmahl}
Let $x$ be a real number. Suppose that for an infinity of different indices $s\ge 1$ the denominators $q_{s-1}$, $q_s$, $q_{s+1}$ of three consecutive convergents of $x$ are divisible by only a finite system of prime numbers. Then $x$ is a Liouville number.
\end{thm} 

\begin{proof}[Proof of Theorem~\ref{nonempty}]
Let $\pi_0$ and $\pi_1$ be two distinct primes in $\Pi$. From Theorem~\ref{erdmahl} and Property~(\ref{approx}), it is enough to prove the existence of uncountably many irrationals for which all the denominators $q_s$ of their convergents are only divisible by $\pi_0$ and $\pi_1$ as soon as $s\ge 1$. To this end, first note that the relationships~(\ref{pn}) and~(\ref{qn}) may be rewritten more succinctly in the form 
$$
\begin{pmatrix}
  p_s & q_s\\
  p_{s-1} & q_{s-1}
\end{pmatrix}
= \prod_{k=0}^{s} 
\begin{pmatrix}
  a_k & 1\\
  1 & 0
\end{pmatrix}
$$
if $s\ge 0$. Taking the determinant on both sides of this equation, it appears on the one hand that any sequence of positive integers $(p_s)_{s\ge 0}$ satisfying~(\ref{pn}) is such that, for any $s\ge 0$, $\gcd(p_s, q_s) =1$ if the sequence $(q_s)_{s\ge 0}$ satisfies~(\ref{qn}). It is therefore enough to solve~(\ref{qn}), where the unknowns are the two sequences $(a_s)_{s\ge 0}$ and $(q_s)_{s\ge 0}$. On the other hand, the same relationship shows that two consecutive denominators $q_s$ and $q_{s+1}$ are coprime~: this leads one to choose all the $q_{2s}$ ($s\ge 0$) as powers of the prime $\pi_0$ and all the $q_{2s+1}$ ($s\ge 0$) as powers of $\pi_1$ (or conversely). 

Thus, let $q_0=1$ and $q_1=a_1 q_0 = \pi_1$. Setting $q_{2s+i} = \pi_{i}^{\alpha_{2s+i}}$ for $s\ge 0$ and $i=0,1$ with $(\alpha_s)_{s\ge 0}$ an increasing sequence of positive integers such that $\alpha_0=0$, Equation~(\ref{qn}) amounts to the following~: for all $s\ge0$ and all $i\in\left\{0,1\right\}$,
\begin{equation*}
a_{2s+i}\, \pi_{i}^{\alpha_{2s-1+i}} = \pi_{1-i}^{\alpha_{2s-2+i}} \left(\pi_{1-i}^{\alpha_{2s+i}-\alpha_{2s-2+i}} -1\right).
\end{equation*}
It should be clear at this stage that both sequences $(a_s)_{s\ge 0}$ and $(q_s)_{s\ge 0}$ shall be uniquely determined if the sequence $(\alpha_s)_{s\ge 0}$ (with $\alpha_0=0$) is chosen in such a way that, for all $s\in\N^*$,
\begin{equation*}
\pi_{1-i}^{\alpha_{2s+i}-\alpha_{2s-2+i}} \equiv 1 \pmod{\pi_i^{\alpha_{2s-1+i}}} 
\end{equation*}
for $i=0,1$. It is easy to construct such a sequence $(\alpha_s)_{s\ge 0}$ (with $\alpha_0=0$) by induction~: first, for two coprime integers $a$ and $b$, denote by $\omega(a,b)$ the order of $b$ in the multiplicative group $\left(\Z/\!\raisebox{-.65ex}{\ensuremath{a\Z}}\right)^{\times}$. Then, choose $\alpha_1 \in \N^*$ and set, for any $s\ge 1$ and $i=0,1$, $$\alpha_{2s+i} = \alpha_{2s-2+i} + k_{2s+i}\,\omega(\pi_{i}^{\alpha_{2s-1+i}},\pi_{1-i}),$$ where at each step the integer $k_{2s+i}$ is chosen such that $\alpha_{2s+i} > \alpha_{2s+i-1}$.

This proves the existence of a Liouville number in $W_{\tau}^*(\Pi)$ ($\tau>2$). Since infinitely many choices are possible for the integer $k_{2s+i}$ at each step, $W_{\tau}^*(\Pi)$ actually contains uncountably many Liouville numbers as soon as $\tau>2$.
\end{proof}

\paragraph{}
To conclude, various issues which have arisen throughout this note are now collected. 

\begin{qst}\label{qst1}
Let $\Pi$ be any \textit{finite} system of prime numbers and let $\tau >2$ be a real number. Define $W_{\tau}^*(\Pi)$ as in~(\ref{wpi}).

Does $W_{\tau}^*(\Pi)$ contain any non--Liouville numbers?
\end{qst}

\begin{qst}\label{qst2}
Let $n\ge 2$ be an integer and let $\tau>1$ be a real number. Let $Q$ be an $\N^*\backslash Q$--free set of integers such that $\nu(Q)=1$. Define $W_{\tau, n}^*(Q)$ as in~(\ref{wtnq}).

Is it true that, if $(x_1, \dots, x_n)\in W_{\tau, n}^*(Q)$, then $x_1$, $x_2$, $\dots$, $x_n$ are $\Q$--linearly dependent?
\end{qst}

Thus, Question~\ref{qst2} amounts to establishing the converse inclusion in~(\ref{inclus}) for suitable values of $A$, $u$ and $v$. This is a problem of particular interest when the set $Q$ is chosen as in~(\ref{qm})~: indeed, the method used by Babai and \v{S}tefankovi\v{c} in~\cite{baba} provides in this case a positive answer to Question~\ref{qst2} if one takes $\epsilon/q$ for a fixed $\epsilon >0$ as the error function in $W_{\tau, n}^*(Q)$ instead of $q^{-\tau}$. However, their method, which consists of investigating certain probability measures on lattices and Fourier transforms of such measures, cannot be extended to a more general class of error functions.

\paragraph{}
Regarding the computation of the Hausdorff dimension of a liminf set such as $W_{\tau, n}^*(Q)$, the main difficulty is to find an upper bound. Theorem~\ref{princi} shows that such a dimension strongly depends on the \textit{size} of $Q$ measured by its exponent of convergence $\nu(Q)$. Some heuristic arguments lead one to state the following two conjectures~:

\begin{conj}\label{conj1}
Under the assumptions of Question~\ref{qst1}, $$\dim W_{\tau}^*(\Pi) = 0.$$
\end{conj}

\begin{conj}\label{conj2}
Under the assumptions of Question~\ref{qst2}, $$\dim W_{\tau, n}^*(Q) = \frac{n}{\tau}$$ when $\tau > 1+1/(n-1).$
\end{conj}

Obviously, a negative (resp. positive) answer to Question~\ref{qst1} (resp. to Question~\ref{qst2}) would imply Conjecture~\ref{conj1} (resp. Conjecture~\ref{conj2}).

\renewcommand{\abstractname}{Acknowledgements}
\begin{abstract}
The author would like to thank his PhD supervisor Detta Dickinson for suggesting the problem and for discussions which helped to develop ideas put forward. He is supported by the Science Foundation Ireland grant RFP11/MTH3084.
\end{abstract}

\bibliographystyle{plain}
\bibliography{Note}

\end{document}